\documentclass[12pt,reqno]{amsart}

\usepackage{amsthm, amsmath, amscd, amsfonts}
\usepackage{cases}
\usepackage{nameref, hyperref, cleveref}
\usepackage{fancyhdr}
\usepackage{array}
\usepackage[normalem]{ulem}
\usepackage{caption}
\usepackage{subcaption}
\usepackage{graphicx}
\usepackage[usenames,dvipsnames]{xcolor}
\usepackage{tikz}
\usepackage[color=WildStrawberry]{todonotes}
\usepackage{geometry}
\usepackage[parfill]{parskip} 
\usepackage{multicol}
\usepackage{mathabx}

\usetikzlibrary{shapes}

\newcommand{\Z}{\mathbb{Z}}

\newcommand{\lcm}{\textup{lcm}}

\renewcommand{\a}{\alpha}
\newcommand{\G}{\mathcal{G}}
\newcommand{\K}{\mathcal{K}}
\newcommand{\h}{\mathcal{H}}

\newcommand{\gthree}{\ell_i\cdot \ell_{n}[\ell_{n}^{-1}]_{\text{mod $\ell_{n-1}$}}}
\newcommand{\Newspline}{Triangulation}
\newcommand{\newspline}{triangulation}
\newtheorem{theorem}{Theorem}[section]
\newtheorem{corollary}[theorem]{Corollary}

\newtheorem{proposition}[theorem]{Proposition}

\newtheorem{definition}[theorem]{Definition}
\newtheorem{remark}[theorem]{Remark}

\newtheorem{question}[theorem]{Question}

\definecolor{mygray}{cmyk}{0.1875,0,0.375,0.8745}
\definecolor{myraspberry}{cmyk}{0,1,0.15,0.3}
\definecolor{mysunshine}{cmyk}{0,0.104,0.542,0.0157}
\definecolor{mypurple}{cmyk}{0.324,0.743,0,0.488}

	\tikzstyle{rectanglevertex}=
	[rounded corners=3pt,
	line width=1pt,
	draw=mygray,
	fill=black!5, 
	font=\fontsize{11}{11}]
	\tikzstyle{vertex}=[rounded corners=3,line width=1pt, draw=black!50,fill=white, font=\fontsize{12}{12}\boldmath\rmfamily\bfseries]
	\tikzstyle{edge}=[line width=1pt,]
	\tikzstyle{edgelabel}=[font=\fontsize{12}{12}\boldmath\rmfamily\bfseries,color=myraspberry]
	\tikzstyle{newedgelabel}=[font=\fontsize{8}{12}\boldmath\rmfamily\bfseries,color=myraspberry]
	\tikzstyle{dashedarrow}=[dashed, line width=4pt, ->]

\tikzstyle{dashededge}=[line width=1.5pt,dotted]

\title[Bases and Structure Constants of Generalized Splines]{Bases and Structure constants of Generalized Splines with Integer Coefficients on Cycles}
\author[Bowden, Hagen, King, and Reinders]{Nealy Bowden, Sarah Hagen, Melanie King, and Stephanie Reinders }
\thanks{ We are extremely grateful to Julianna Tymoczko, Elizabeth Drellich, and Yue Cao for their insight and contributions to this paper. We would also like to thank Ruth Haas and Joshua Bowman for valuable discussions on these topics, and Michael DiPasquale for his thorough review and comments. This work was supported by Smith College and the National Science Foundation through the Center for Women in Mathematics [DMS-1143716]. }

\begin{document}

\maketitle

\begin{abstract}

An {\it integer generalized spline} is a set of vertex labels on an edge-labeled graph that satisfy the condition that if two vertices are joined by an edge, the vertex labels are congruent modulo the edge label. Foundational work on these objects comes from Gilbert, Polster, and Tymoczko, who generalize ideas from geometry/topology (equivariant cohomology rings) and algebra (algebraic splines) to develop the notion of {\it generalized splines}. Gilbert, Polster, and Tymoczko prove that the ring of splines on a graph can be decomposed in terms of splines on its subgraphs (in particular, on trees and cycles), and then fully analyze splines on trees. Following Handschy-Melnick-Reinders and Rose, we analyze splines on cycles, in our case integer generalized splines. 

The primary goal of this paper is to establish two new bases for the module of integer generalized splines on cycles: the triangulation basis and the King basis. Unlike bases in previous work, we are able to characterize each basis element completely in terms of the edge labels of the underlying cycle. As an application we explicitly construct the multiplication table for the ring of integer generalized splines in terms of the King basis.

\end{abstract}


\section{Introduction}

An {\it integer generalized spline} is a set of vertex labels on an edge-labeled graph that satisfy the condition that if two vertices are joined by an edge, the vertex labels are congruent modulo the edge label. (See Definition \ref{defSplines} for a precise statement.) Figure \ref{Three-cycle} shows examples of splines on a three-cycle. 
\begin{figure}
\begin{tikzpicture}
	
	\begin{scope}
		\pgfmathsetmacro{\r}{1.25}

		\draw[edge] (-90:\r)--(0:\r);
		\draw[edge] (0:\r)--(90:\r);
		\draw[edge] (90:\r)--(-90:\r);		
	
		\node[edgelabel] at (-45:\r) {$2$};
		\node[edgelabel] at (45:\r) {$5$};
		\node[edgelabel] at (180:\r /4) {$3$};
		
		\node[vertex] at (-90:\r) {$1$};
		\node[vertex] at (0:\r) {$1$};
		\node[vertex] at (90:\r) {$1$};
	\end{scope}

	\begin{scope}[xshift=1.5 in]
		\pgfmathsetmacro{\r}{1.25}

		\draw[edge] (-90:\r)--(0:\r);
		\draw[edge] (0:\r)--(90:\r);
		\draw[edge] (90:\r)--(-90:\r);		
	
		\node[edgelabel] at (-45:\r) {$2$};
		\node[edgelabel] at (45:\r) {$5$};
		\node[edgelabel] at (180:\r /4) {$3$};
		
		\node[vertex] at (-90:\r) {$0$};
		\node[vertex] at (0:\r) {$2$};
		\node[vertex] at (90:\r) {$12$};
	\end{scope}

	\begin{scope}[xshift=3 in]
		\pgfmathsetmacro{\r}{1.25}

		\draw[edge] (-90:\r)--(0:\r);
		\draw[edge] (0:\r)--(90:\r);
		\draw[edge] (90:\r)--(-90:\r);		
	
		\node[edgelabel] at (-45:\r) {$2$};
		\node[edgelabel] at (45:\r) {$5$};
		\node[edgelabel] at (180:\r /4) {$3$};
		
		\node[vertex] at (-90:\r) {$0$};
		\node[vertex] at (0:\r) {$0$};
		\node[vertex] at (90:\r) {$15$};
	\end{scope}

\end{tikzpicture}
\caption{The edge labels are $\{2,5,3\}$ and the sets of vertex labels $\{1,1,1\}$, $\{0,2,12\}$, and $\{0,0,15\}$ each form a spline on the cycle.}
\label{Three-cycle}
\end{figure}
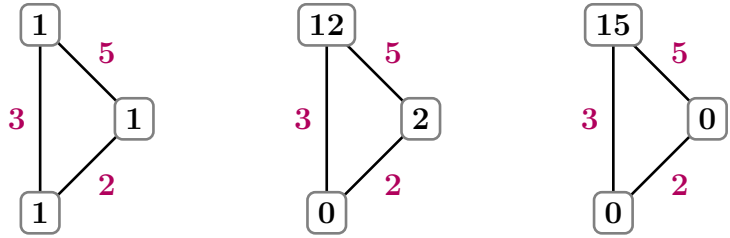

The term ``spline'' comes from the name of the thin strips of wood used by engineers to model larger constructions like ships or cars. Mathematicians later adopted the term to refer to piecewise polynomials on polytopes with the property that the polynomials on the faces agree at their shared edges up to a given degree of smoothness. These mathematical splines are also used for object-modeling purposes, hence the use of the name.

Billera pioneered the algebraic study of splines, especially looking into questions regarding the dimension of the module of splines \cite{Billera88}. Many people continued Billera's work, including among others, Rose \cite{Rose95,Rose04} and Haas \cite{Haas91} who worked on identifying dimension and bases for the module of splines. 

Spline theory developed independently in topology and geometry. Goresky, Kottwitz, and MacPherson \cite{GKM}, Payne \cite{Payne06}, and Bahri, Franz, and Ray \cite{Bahri09} constructed equivariant cohomology rings using splines, although they did not use that name. 

Gilbert, Polster, and Tymoczko generalize the notion of splines that we use here to what they call \emph{generalized splines} \cite{Gilbert} . These generalized splines are built on the dual graph of the polytopes found in classical splines. The work of Billera and Rose shows that the two constructions (on polytopes or their duals) are equivalent in most cases, including the cases of classical interest \cite{Billera91-2}.

Cycles turn out to be a particularly important family of graphs to study. Indeed Gilbert, Polster, and Tymoczko show that the ring of generalized splines on a graph $G$ can be decomposed in terms of splines on certain trees and cycles in $G$ \cite{Gilbert}. They completely describe splines on trees, while leaving open the investigation of splines on cycles. Similarly, Rose showed that cycles play a key role in the relations defining modules of splines \cite{Rose04}. 

Handschy, Melnick, and Reinders begin analysis of integer generalized splines on cycles \cite{HMR}. They prove the existence of a certain flow-up basis (see Definition \ref{flowupdef}), what we call the smallest-value basis, for splines on cycles, and thus prove that such spline modules are free. They define their basis for arbitrary cycles, but only have formulas for the leading nonzero elements. 

 In this paper we introduce two new bases for the module of integer generalized splines on cycles: the triangulation basis and the King basis. Each of these bases is fully expressible in terms of the edge labels of the cycle, and each has its own strengths. The triangulation basis, so called because it is constructed from triangulated cycles, is useful because it exists on arbitrary cycles (Theorem \ref{triangbasis}). The advantage of the King basis lies in the fact that it is relatively simple to calculate, with the entries almost constant (Definition \ref{KingSplines}). Although the King basis only exists on cycles with a pair of relatively prime adjacent edge labels, this restriction is not uncommon in applications. In fact an even greater restriction that all edge labels be relatively prime is commonly used \cite{Goldin09,Knutson}. The results of our work naturally generalize to principle ideal domains, which include classical univariate splines and Pr\"{u}fer domains; see forthcoming work \cite{HT}.
 
 As an application we present the multiplication table of splines on cycles where the products of splines are expressed in terms of the King basis. Finding multiplication tables of equivariant cohomology rings in terms of Schubert bases is the central problem of Schubert calculus. We view this work as a step in that geometric direction.
 
 The rest of this paper is organized as follows. In Section \ref{prelim} we summarize the important definitions and theorems that we use in our work. In Section \ref{bcsec} we provide a criterion for the existence of flow-up bases. Sections \ref{triangsec} and \ref{kingsec} are dedicated to proving the existence of the triangulation basis and King basis respectively. In the final section we give the multiplication table for the King basis and end with an open question.


\section{Preliminaries}\label{prelim}

\subsection{Results from Handschy, Melnick, and Reinders}
Handschy, Melnick, and Reinders proved a number of results about splines on cycles \cite{HMR}. Many of their propositions and theorems play key roles in our proofs regarding \newspline \ splines and King splines. We also use their notation, which we describe in this section.

\subsubsection{Basic Definitions}
The foundational combinatorial object we study is an edge-labeled graph, defined here:
\begin{definition}[Edge-Labeled Graphs]
\label{defSplines}
	Let $G$ be a graph with $k$ edges ordered $e_1, e_2, \dots, e_k$ and $n$ vertices ordered $v_1,...,v_n$.   Let $\ell_i$ be a positive integer label on edge $e_i$ and let $L=\{\ell_1,...,\ell_k\}$ be the set of edge labels. Then $(G,L)$ is an edge-labeled graph.  
\end{definition}

With this notation for edge-labeled graphs we have the formal definition of splines:
\begin{definition}[Splines]
	A spline on the edge-labeled graph $(G,L)$ is a vertex-labeling as follows: if two vertices are connected by an edge $e_i$ then the two vertex labels are equivalent modulo $\ell_i$. We denote a spline $\G=(g_1,...,g_n)$ where $g_i$ is the label on vertex $v_i$ for $1 \leq i \leq n$.
\end{definition}

In this paper we assume the labels $g_i \in \Z$.

\subsubsection{Flow-Up Classes and the Smallest-Value Basis}

Flow-up classes are a particularly nice class of splines on cycles. They arise geometrically (\cite{Goldin09}, \cite{Knutson}, \cite{Tym05}) and are an analogue of upper triangular matrices. 

\begin{definition}[Flow-Up Classes]\label{flowupdef}
	 Fix a cycle with edge labels $(C_n,L)$ and fix $k$ with $1\leq k<n$. A flow-up class $\G_k$ on $(C_n,L)$ is a spline with $k$ leading zeros. 
\end{definition}

We say that a basis whose elements are flow-up classes is a \emph{flow-up basis}. The simplest flow-up class is the trivial spline; It exists on any edge-labeled cycle.

\begin{proposition}[Trivial Splines {\cite[Prop 2.5]{HMR}}] 
\label{TrivialSplines}
	Fix a cycle with edge labels $(C_n,L)$. The smallest flow-up class  on $(C_n,L)$ is $\G_0=(1,...,1)$. Moreover, any multiple of $\G_0$  is also a spline. We call the multiples of $\G_0$ trivial splines.
\end{proposition}

The following theorem establishes that flow-up classes exist on any edge-labeled cycle. 

\begin{theorem}[Flow-Up Classes on $n$-cycles {\cite[Thrm 4.3]{HMR}}]
\label{FlowUpClassesExist}
	Fix a cycle with edge labels $(C_n,L)$. Let $n\geq 3$ and $1 \leq k< n$. There exists a flow-up class $\G_k$ on $(C_n,L)$.
\end{theorem}

The next definition introduces smallest flow-up classes. 

\begin{definition}[Smallest Flow-Up Class]
	Fix a cycle with edge labels $(C_n,L)$. The smallest flow-up class $\G_k=(0,...,0,g_{k+1},...,g_n)$ on $(C_n,L)$ is the flow-up class whose nonzero entries are positive and if $\G_k'=(0,...,0,g_{k+1}',...,g_n')$ is another flow-up class with positive entries then $g_i' \geq g_i$ for all entries. By convention we consider \\ $\G_0=(1,...,1)$ the smallest flow-up class $\G_0$.
\end{definition}

The following theorem gives an explicit formula for the smallest leading element of flow-up classes.

\begin{theorem}[Smallest Leading Element of $\G_k$ {\cite[Thrm 4.5]{HMR}}] 
\label{SmallestElemn}
Fix a cycle with edge labels $(C_n,L)$. Fix $n\geq 3$ and $k$ such that $2 \leq k <n$. Let  $\G_{k-1}=(0,...,0,g_k...,g_n)$ be a flow-up class on $(C_n,L)$. The leading element $g_k$ is a multiple of $\lcm(\ell_{k-1},\gcd(\ell_k,...,\ell_n))$ and there is a flow-up class $\G_{k-1}$ with  $g_k=\lcm(\ell_{k-1},\gcd(\ell_k,...,\ell_n))$.
\end{theorem}

The smallest flow-up classes exist and form a basis for the set of splines given any edge-labeled cycle.

\begin{theorem}[Basis for $n$-Cycles {\cite[Thrm 4.7]{HMR}}]
\label{HMRbasis}
	Fix a cycle with edge labels $(C_n,L)$. The smallest flow-up classes $\G_0, \G_1, \dots, \G_{n-1}$  exist on $(C_n,L)$ and form a basis over the integers for the $\Z$-module of splines  on $(C_n,L)$.
\end{theorem}

\subsection{Useful Computational Tool}

For reasons related to finding an explicit basis for splines on cycles, we want to find a formula for the value of the variable $x$ in the following pair of congruences:
\[ \begin{cases}
x \equiv y \bmod a \\
x \equiv 0 \bmod b \\
\end{cases} \]
We note the conditions for when such a solution exists and we give an explicit formulation for $x$ in terms of $y$, $a$, and $b$ provided a solution does exist.

\begin{proposition}\label{modprop}
Consider the system of congruences
\[ \begin{cases}
x \equiv y \bmod a \\
x \equiv 0 \bmod b. \\
\end{cases} \]
If this system has a solution then one solution is given by the following formula:
\begin{itemize}
\item If $\frac{a}{\gcd(a,b)}=1$ then $x=b$ is a solution to the system. 
\item If $\frac{a}{\gcd(a,b)}\neq 1$ then
\[ x = y \left(\frac{b}{\gcd(a,b)}\right)\left(\frac{b}{\gcd(a,b)}\right)^{-1}_{\bmod  \left(\frac{a}{\gcd(a,b)} \right)} \]
is a solution to the system.
\end{itemize}
\end{proposition}

\begin{proof}
The Chinese Remainder Theorem tells us that this system of congruences is satisfied if and only if $y\equiv 0 \bmod \gcd(a,b)$. In what follows we will assume that a solution exists, and thus that $y\equiv 0 \bmod \gcd(a,b)$.

\emph{Case 1:} Let's deal first with the case where $\frac{a}{\gcd(a,b)}=1$. This condition implies that  $\gcd(a,b)=a$ and so $b=an$ for some $n\in \Z$. Because $y\equiv 0 \bmod \gcd(a,b)$ by assumption and $\gcd(a,b)=a$ we have $y\equiv 0 \bmod a$. In other words, $y= am$ for some $m\in \Z$. Then $x=b$ satisfies the system of congruences because $b$ is congruent to zero modulo $b$ and $b=an$  is congruent to $y=am$ modulo $a$.

\emph{Case 2:} Now suppose $\frac{a}{\gcd(a,b)}\neq 1$. We can rewrite the system of congruences as
\[ \begin{cases}
x = y + as \\
x = bt \\
\end{cases} \]
Equate both expressions.
\[ bt = y+as \]
Recall that $y\equiv 0 \bmod \gcd(a,b)$. This allows us to divide both sides by $\gcd(a,b)$ and get an integer as the result. 
\[ \left(\frac{b}{\gcd(a,b)}\right) t = \frac{y}{\gcd(a,b)} + \left(\frac{a}{\gcd(a,b)}\right)s \]
Putting this back into modular form we have 
\[ \left(\frac{b}{\gcd(a,b)}\right) t =\frac{y}{\gcd(a,b)} \bmod \left(\frac{a}{\gcd(a,b)}\right). \]
The integers $ \left(\frac{b}{\gcd(a,b)}\right)$ and $\left(\frac{a}{\gcd(a,b)}\right)$ are relatively prime so we can take the inverse of the first modulo the second. 
\[ t \equiv \frac{y}{\gcd(a,b)}\left(\frac{b}{\gcd(a,b)}\right)^{-1} \bmod  \left(\frac{a}{\gcd(a,b)} \right). \]
Plug this expression for $t$ into the equation $x=bt$: 
\[ x = y \left(\frac{b}{\gcd(a,b)}\right)\left(\frac{b}{\gcd(a,b)}\right)^{-1}_{\bmod  \left(\frac{a}{\gcd(a,b)} \right)}. \]
This value is a solution to the original system of congruences. 
\end{proof}

Notice that this second case simplifies enormously if  $gcd(a,b)=1$. In this situation $x$ reduces to: $$x = yb[b^{-1}]_{\text{mod $a$}} $$


\section{Basis Condition}\label{bcsec}

Let $(G,L)$ be an arbitrary graph on $n$ vertices with an arbitrary edge-labeling. Consider a set of flow-up classes $\mathcal{G}_0 \ldots \mathcal{G}_{n-1}$  on $(G,L)$. In this section we give a necessary and sufficient condition for this set to form a basis for the module of the splines on $(G,L)$.  Any set $\mathcal{G}_0, \ldots, \mathcal{G}_{n-1}$ that meets this basis condition is called a \emph{flow-up basis}. Such a basis is  useful because linear independence is trivially verified.

Let $\mathcal{G}_0 \ldots \mathcal{G}_{n-1}$ be a set of flow-up classes and for each $i$ denote $$\G_i = (0,\ldots, 0, g_{i+1}^{(i)},\ldots, g_n^{(i)}).$$ The subscript of each $g^{(i)}$ indicates the entry-position of $g^{(i)}$ in the spline $\G_i$. The superscript $(i)$ is to keep track of the fact that we are working with the flow-up class $\G_i$. In much of this paper and in previous work the superscript is suppressed when the flow-up class in question is obvious.

\begin{theorem}[Basis Condition]
The following are equivalent: 
\begin{itemize}
\item The set $\{\G_0,\ldots,\G_{n-1}\}$ forms a flow-up basis.
\item For each flow-up spline  $A_i = (0,\ldots,0,a_{i+1},\ldots,a_{n})$ the entry $a_{i+1}$ of $A_i$ is an integer multiple of the entry $g_{i+1}^{(i)}$ of $\G_i$.
\end{itemize}
\end{theorem}

 \begin{proof} 
Suppose that $\G_0,\ldots,\G_{n-1}$ forms a flow-up basis for the module of splines on a graph $(G,L)$. Suppose that $A_i = (0,\ldots,0,a_{i+1},\ldots,a_{n})$ is a spline on $(G,L)$ with exactly $i$ leading zeros. We will show that $a_{i+1} = cg_{i+1}^{(i)}$ for some $c \in \Z$.

 Since $\G_0,\ldots,\G_{n-1}$ form a basis, we can write $A_i$ as a linear combination of the splines $\G_0,\ldots,\G_{n-1}$. The fact that $A_i$ has $i$ leading zeros implies that the coefficients of $\G_0,\ldots,\G_{i-1}$ must be 0. Thus we have $A_i = c_i\G_i + \ldots + c_{n-1}\G_{n-1}$ for some \mbox{$c_i,\ldots,c_{n-1} \in \Z$}. Consider the $(i+1)^{th}$ entry of the splines on the right-hand side of this equation. Note that $\G_i$ is the only element of $\G_i,\ldots,\G_{n-1}$ with a nonzero entry in this position. Considering the $(i+1)^{th}$ entry on each side of the equation, we have $$a_{i+1} = c_ig_{i+1}^{(i)} + c_{i+1}0 + \ldots + c_{n-1}0 = c_ig_{i+1}^{(i)}.$$   

Now we prove the converse. Let $A = (a_1,\ldots, a_n)$ be an arbitrary spline on $(G,L)$. We prove by induction that $$A = A'_j + \sum_{k=0}^{j-1}c_k \G_k$$  for all $1 \leq j \leq n$ where $A'_j$ is a spline with (at least) $j$ leading zeros.

For our base case, note that by hypothesis we have 
$$A = \left(\begin{array}{c}a_{n} -c_0g_{n}^{(0)}\\ \vdots \\ a_2 - c_0g_{2}^{(0)}\\0\end{array}\right) + c_0\G_0$$ since $a_1 = c_0g_{1}^{(0)}$. Letting $A'_1 = (0, a_2 - c_0g_{2}^{(0)},\ldots,a_{n} -c_0g_{n}^{(0)})$ gives \mbox{$A = A'_1+ \sum_{k=0}^{0}c_k \G_k$}. Thus  our claim holds for $j=1$.
 
Suppose as our induction hypothesis that we have $A = A'_i + \sum_{k=0}^{i-1}c_k \G_k$ for some $1 \leq i \leq n-1$. We can write this as $$A = \left(\begin{array}{c}a_{n}' \\ \vdots \\a_{i+1}'\\0\\ \vdots \\0\end{array}\right) + \sum_{k=0}^{i-1}c_k \G_k.$$ 
By hypothesis we have that $a_{i+1}' = c_i g_{i+1}^{(i)}$ for some $c_i \in \Z$. So we can write 
$$A = \left(\begin{array}{c}a_{n}' -c_i g_{n}^{(i)}\\ \vdots \\ a_{i+2}' - c_i g_{i+2}^{(i)}\\ 0 \\ 0\\ \vdots \\0\end{array}\right) +  \sum_{k=0}^{i}c_k \G_k.$$ 
Letting $A'_{i+1} = (0,\ldots,0,0,a_{i+2}' - c_i g_{i+2}^{(i)},\ldots,a_{n}' -c_i g_{n}^{(i)})$ gives us $A = A'_{i+1} +  \sum_{k=0}^{i}c_k \G_k.$

By induction we have $A = A'_j + \sum_{k=0}^{j-1}c_k \G_k$ for all $1 \leq j \leq n$. In particular we have $A = A'_n + \sum_{k=0}^{n-1}c_k \G_k$. But $A'_n$ is a spline with $n$ leading zeros. So $A'_n = (0,\ldots,0)$.  Thus $A =  \sum_{k=0}^{n-1}c_k \G_k$. We conclude that every spline can be written as a linear combination of $\G_0,\ldots, \G_{n-1}$ as desired.
\end{proof}

One important observation is that the basis condition is only a condition on the first nonzero entry of each spline in a set of flow-up classes $\G_0,\ldots, \G_{n-1}$. This gives us the following useful corollary:
\begin{corollary}\label{bccor}
Suppose the set of flow-up classes $\{\G_0,\ldots, \G_{n-1}\}$ forms a basis for the module of splines. Suppose $\{\G'_0,\ldots, \G'_{n-1}\}$ is a set of flow-up classes for which for each $i$ the first nonzero entry of $\G'_i$ equals the first nonzero entry of $\G_i$. Then the set $\{\G'_0,\ldots, \G'_{n-1}\}$ also forms a basis for the module of splines.
\end{corollary}


\section{The Triangulation Splines}\label{triangsec}

\Newspline \ splines form another basis of flow-up classes for cycles. They are similar to Handschy, Melnick, and Reinders' smallest-value flow-up classes in that the leading nonzero elements of both are the same.  However we give a formula for every entry of the \newspline \  splines, unlike the smallest-value flow-up classes.

\begin{definition}[\Newspline \ Splines]\label{GSFUCdef} 
Fix an edge-labeled cycle $(C_n,L)$. For \mbox{$1 \leq k \leq n-1$} the vector $\h_k = (0,...,0,h_{k+1},...,h_n)$  has entries as follows:
\begin{itemize} 
\item $h_{k+1} = \lcm(\ell_k,\gcd(\ell_{k+1},...,\ell_n))$ 
\item For $k+1 < i \leq n$ if $\frac{\ell_{i-1}}{\gcd(\ell_{i-1},...,\ell_n)}=1$ then $h_i = \gcd(\ell_i,...,\ell_n)$. 
\item For $k+1 <  i \leq n$ if $\frac{\ell_{i-1}}{\gcd(\ell_{i-1},...,\ell_n)}\neq 1$ then
\[ h_i =  h_{i-1} \left(\frac{\gcd(\ell_i,...,\ell_n)}{\gcd(\ell_{i-1},...,\ell_n)}\right)\left(\frac{\gcd(\ell_i,...,\ell_n)}{\gcd(\ell_{i-1},...,\ell_n)}\right)^{-1}_{\bmod \frac{\ell_{i-1}}{\gcd(\ell_{i-1},...,\ell_n)}} \]
\end{itemize}
\end{definition}

The next theorem establishes that \newspline \ splines exist on any edge-labeled cycle. 

\begin{theorem}[Existence of \Newspline \ Splines]\label{triangbasis}
Fix an edge-labeled cycle $(C_n,L)$. For $1\leq k \leq n-1$ the vector  $\mathcal{H}_k$ is a spline on $(C_n,L)$. 
\end{theorem}

\begin{proof}

Start with an edge-labeled cycle $(C_n,L)$. For $3 \leq k \leq n-1$  add an edge between vertices $v_1$ and $v_k$ as shown in Figure \ref{triang-cycle}. Label the edge between $v_1$ and $v_k$ with $\gcd(\ell_k,...,\ell_n)$. We will show the vector $\h_k$ satisfies all of the edge conditions represented by this graph, which implies it satisfies the cycle's edge conditions in particular.

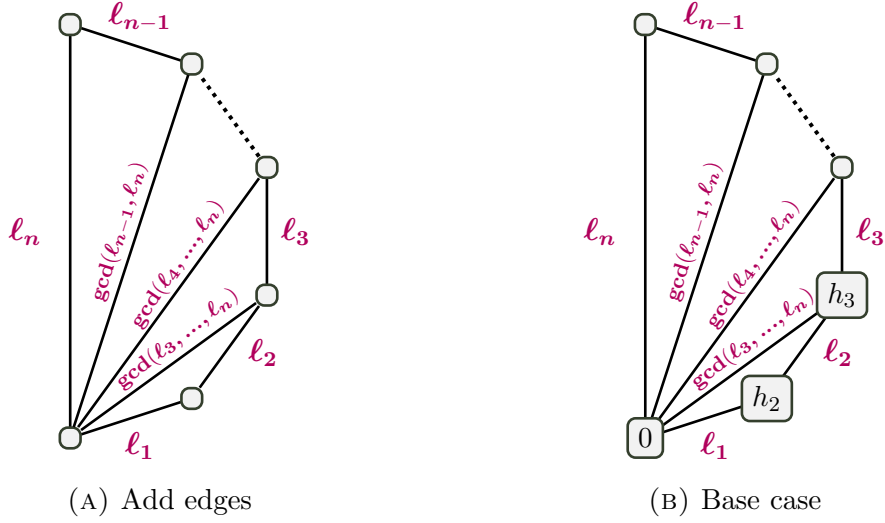
\begin{figure}[t!]
\centering
\begin{subfigure}[t]{0.5\textwidth}
\centering
\begin{tikzpicture}
\begin{scope}
	\pgfmathsetmacro{\r}{2.75}
	\pgfmathsetmacro{\ro}{3}
	\pgfmathsetmacro{\edge}{36}

	\node (a) at ({-90 + \edge*0}:\r) {};
	\node (b) at ({-90 + \edge*1}:\r) {};
	\node (c) at ({-90 + \edge*2}:\r) {};
	\node (d) at ({-90 + \edge*3}:\r) {};
	\node (e) at ({-90 + \edge*4}:\r) {};
	\node (f) at ({-90 + \edge*5}:\r) {};

	\draw[edge] (a)--(b);
	\draw[edge] (b)--(c);
	\draw[edge] (c)--(d);
	\draw[dashededge] (d)--(e);
	\draw[edge] (e)--(f);
	\draw[edge] (f)--(a);

	\node[edgelabel] at ({-72 + \edge*0}:\ro) {$\ell_1$};
	\node[edgelabel,right] at ({-72 + \edge*1}:\r) {$\ell_2$};
	\node[edgelabel] at ({-72 + \edge*2}:\ro) {$\ell_3$};
	\node[edgelabel] at ({-72 + \edge*4}:\ro) {$\ell_{n-1}$};
	\node[edgelabel] at (180:\ro /5) {$\ell_n$};

	\node[rectanglevertex] at (a) {$$};
	\node[rectanglevertex] at (b) {$$};
	\node[rectanglevertex] at (c) {$$};
	\node[rectanglevertex] at (d) {$$};
	\node[rectanglevertex] at (e) {$$};
	\node[rectanglevertex] at (f) {$$};

	\draw[edge] (a)--(c);
	\draw[edge] (a)--(d);
	\draw[edge] (a)-- (e);

	\node[rotate=73,newedgelabel] at (0.65,0) {$\gcd(\ell_{n-1},\ell_n)$};
	\node[rotate=56,newedgelabel] at (1.43,-.43) {$\gcd(\ell_4,...,\ell_n)$};
	\node[rotate=38,newedgelabel] at (1.43,-1.43) {$\gcd(\ell_3,...,\ell_n)$};

\end{scope}
\end{tikzpicture}
\caption{Add edges}
\end{subfigure}
~
\begin{subfigure}[t]{0.5\textwidth}
\centering
\begin{tikzpicture}

\begin{scope}
	\pgfmathsetmacro{\r}{2.75}
	\pgfmathsetmacro{\ro}{3}
	\pgfmathsetmacro{\edge}{36}

	\node (a) at ({-90 + \edge*0}:\r) {};
	\node (b) at ({-90 + \edge*1}:\r) {};
	\node (c) at ({-90 + \edge*2}:\r) {};
	\node (d) at ({-90 + \edge*3}:\r) {};
	\node (e) at ({-90 + \edge*4}:\r) {};
	\node (f) at ({-90 + \edge*5}:\r) {};

	\draw[edge] (a)--(b);
	\draw[edge] (b)--(c);
	\draw[edge] (c)--(d);
	\draw[dashededge] (d)--(e);
	\draw[edge] (e)--(f);
	\draw[edge] (f)--(a);

	\node[edgelabel] at ({-72 + \edge*0}:\ro) {$\ell_1$};
	\node[edgelabel,right] at ({-72 + \edge*1}:\r) {$\ell_2$};
	\node[edgelabel] at ({-72 + \edge*2}:\ro) {$\ell_3$};
	\node[edgelabel] at ({-72 + \edge*4}:\ro) {$\ell_{n-1}$};
	\node[edgelabel] at (180:\ro /5) {$\ell_n$};

	\draw[edge] (a)--(c);
	\draw[edge] (a)--(d);
	\draw[edge] (a)-- (e);

	\node[rotate=73,newedgelabel] at (0.65,0) {$\gcd(\ell_{n-1},\ell_n)$};
	\node[rotate=56,newedgelabel] at (1.43,-.43) {$\gcd(\ell_4,...,\ell_n)$};
	\node[rotate=38,newedgelabel] at (1.43,-1.43) {$\gcd(\ell_3,...,\ell_n)$};

	\node[rectanglevertex] at (a) {$0$};
	\node[rectanglevertex] at (b) {$h_2$};
	\node[rectanglevertex] at (c) {$h_3$};
	\node[rectanglevertex] at (d) {$$};
	\node[rectanglevertex] at (e) {$$};
	\node[rectanglevertex] at (f) {$$};

\end{scope}

\end{tikzpicture}
\caption{Base case}
\end{subfigure}
\caption{Triangulated Cycle}
\label{triang-cycle}
\end{figure}

Label vertices $v_1,...,v_k$ zero. Label vertex $v_{k+1}$ with 
\[ h_{k+1} = \lcm(\ell_k,\gcd(\ell_{k+1},...,\ell_n)). \]
The integer $h_{k+1}$ satisfies the edge conditions on the downward edges (edges with lower-indexed vertices) at vertex $v_{k+1}$ by construction:  
\[ \begin{cases}
h_{k+1} \equiv 0 \bmod \ell_k \\
h_{k+1} \equiv 0 \bmod \gcd(\ell_{k+1},...,\ell_n) \\
\end{cases} \]
This is our base case, and we will label vertices from $h_{k+2}$ to $h_{n-1}$ inductively.

Our induction hypothesis is that $h_{k+1},...,h_i$ for $k+1 \leq i \leq n-1$ satisfy the edge conditions for downward edges. Consider the system of congruences at vertex $v_{i+1}$ represented by the edges labeled $\ell_i$ and $\gcd(\ell_{i+1},...,\ell_n)$:
\[ \begin{cases}
h_{i+1} \equiv h_i \bmod \ell_i \\
h_{i+1} \equiv 0 \bmod \gcd(\ell_{i+1},...,\ell_n) \\
\end{cases} \] 
By the Chinese Remainder Theorem a solution $h_{i+1}$ exists if and only if $h_i \equiv 0 \bmod \gcd(\ell_i, \gcd(\ell_{i+1},...,\ell_n))$. In other words a solution exists if and only if $h_i \equiv 0 \bmod \gcd(\ell_i,...,\ell_n)$. By our induction hypothesis $h_i$ satisfies the downward edge conditions at vertex $v_i$ so in particular $h_i \equiv 0 \bmod \gcd(\ell_i,...,\ell_n)$. Thus a solution $h_{i+1}$ exists. 
This means 
\[ h_{i+1}=  \begin{cases} 
 h_{i} \left(\frac{\gcd(\ell_{i+1},...,\ell_n)}{\gcd(\ell_{i},...,\ell_n)}\right)\left(\frac{\gcd(\ell_{i+1},...,\ell_n)}{\gcd(\ell_{i},...,\ell_n)}\right)^{-1}_{\bmod \frac{\ell_{i}}{\gcd(\ell_{i+1},...,\ell_n)}} & \text{ if } \frac{\ell_i}{\gcd(\ell_i,...,\ell_n)}\neq 1 \\
 \gcd(\ell_{i+1},...,\ell_n) & \text{ if }  \frac{\ell_i}{\gcd(\ell_i,...,\ell_n)} = 1
 \end{cases}
 \]
is a solution by Proposition \ref{modprop}. 

In conclusion we can label each vertex $v_i$ for $k+1 < i \leq n-1$ with 
\[ h_{i} = \begin{cases}
  h_{i-1} \left(\frac{\gcd(\ell_{i},...,\ell_n)}{\gcd(\ell_{i-1},...,\ell_n)}\right)\left(\frac{\gcd(\ell_{i},...,\ell_n)}{\gcd(\ell_{i-1},...,\ell_n)}\right)^{-1}_{\bmod \frac{\ell_{i-1}}{\gcd(\ell_{i-1},...,\ell_n)}} &  \text{ if } \frac{\ell_{i-1}}{\gcd(\ell_{i-1},...,\ell_n)}\neq 1 \\
   \gcd(\ell_i,...,\ell_n) &  \text{ if } \frac{\ell_{i-1}}{\gcd(\ell_{i-1},...,\ell_n)} = 1
  \end{cases}
  \]
and $h_i$ will satisfy the edge conditions represented by the edges labeled $\ell_{i-1}$ and $\gcd(\ell_i,...,\ell_n)$.

Lastly for an integer $h_n$ to satisfy the edge conditions at vertex $v_n$ it must satisfy the following system of congruences:
\[ \begin{cases}
h_n \equiv h_{n-1} \bmod \ell_{n-1} \\
h_n \equiv 0 \bmod \ell_n \\
\end{cases} \]
The Chinese Remainder Theorem tells us that a solution $h_n$ exists to this system if and only if $h_{n-1} \equiv 0 \bmod \gcd(\ell_{n-1},\ell_n)$. We showed by induction that our choice of $h_{n-1}$ satisfies the edge conditions of the downward edges at the $(n-1)$-th vertex. In particular this means $h_{n-1} \equiv 0 \bmod \gcd(\ell_{n-1},\ell_n)$ because this is the edge condition represented by the edge labeled $\gcd(\ell_{n-1},\ell_n)$. 
Therefore 
\[ h_{n}=  \begin{cases} 
h_{n-1} \left(\frac{\ell_n}{\gcd(\ell_{n-1},\ell_n)}\right)\left(\frac{\ell_n}{\gcd(\ell_{n-1},\ell_n)}\right)^{-1}_{\bmod \frac{\ell_{n-1}}{\gcd(\ell_{n-1},\ell_n)}} &  \text{ if } \frac{\ell_{n-1}}{\gcd(\ell_{n-1},\ell_n)}\neq 1\\
\ell_n &  \text{ if } \frac{\ell_{n-1}}{\gcd(\ell_{n-1},\ell_n)}= 1
\end{cases} \]
satisfies the vertex $v_n$ edge conditions by Proposition \ref{modprop}. Choose this integer to label the $n$-th vertex.

All of the congruences represented by the graph are accounted for so the vector $\h_k=(0,...,0,h_{k+1},...,h_n)$ is a spline on the graph. In particular $\h_k$ is a spline on the cycle $(C_n,L)$ as desired.

\end{proof} 

The Corollary to the Basis Condition Theorem allows us to succinctly conclude that the set of \newspline \ splines $\h_0,...,\h_{n-1}$ forms a basis for the set of splines on an edge-labeled cycle. 

\begin{theorem}\label{GSFUCthm}
Fix an edge-labeled cycle $(C_n,L)$. The set of \newspline \ splines $\h_0,...,\h_{n-1}$ form a basis for the set of splines on $(C_n,L)$. 
\end{theorem}

\begin{proof}
The set of smallest flow-up classes $\G_0,...,\G_{n-1}$ form a basis for the set of splines on $(C_n,L)$ by Theorem \ref{HMRbasis}. The leading entry of $\h_k$ equals the leading entry of $\G_k$ by construction  for $0 \leq k \leq n-1$. Thus the set of \newspline \ splines $\h_0,...,\h_k$ forms a basis for the set of splines on $(C_n,L)$ by Corollary \ref{bccor}. 
\end{proof}

As an example, we calculate the triangulation basis for the $4$-cycle with edge labels $\{2,6,10,15\}$. 
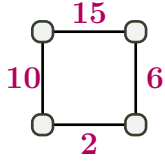
\begin{center}
\begin{tikzpicture}
	\begin{scope}[scale=.7]
		\pgfmathsetmacro{\r}{1.25}
		
		\draw[edge] (-135:\r)--(-45:\r);
		\draw[edge] (-45:\r)--(45:\r);
		\draw[edge] (45:\r)--(135:\r);
		\draw[edge] (135:\r) -- (-135:\r);

		\node[edgelabel] at (-90:\r) {$2$};
		\node[edgelabel] at (0:\r) {$6$};
		\node[edgelabel] at (90:\r) {$15$};
		\node[edgelabel] at (180:\r) {$10$};
	
		\node[rectanglevertex] at (-135:\r) {$$};
		\node[rectanglevertex] at (-45:\r) {$$};
		\node[rectanglevertex] at (45:\r) {$$};
		\node[rectanglevertex] at (135:\r) {$$};
	\end{scope}
\end{tikzpicture}
\end{center}
The first basis element $\h_0$ is, as always, the trivial spline $(1,1,1,1)$. The nonzero entries of the second basis element $\h_1$ are calculated as follows: 
\begin{flalign*}
h_2^{(1)} &=\lcm(2,\gcd(6,10,15))=2 \\
h_3^{(1)} & =2\left(\frac{\gcd(15,10)}{\gcd(6,15,10)}\right)\left(\frac{\gcd(15,10)}{\gcd(6,15,10)}\right)^{-1}_{\bmod \frac{6}{\gcd(6,15,10)}}=2\cdot5\cdot(5)^{-1}_{\bmod 6}=50 \\
h_4^{(1)} &=50\left(\frac{\gcd(10)}{\gcd(15,10)}\right)\left(\frac{\gcd(10)}{\gcd(15,10)}\right)^{-1}_{\bmod \frac{15}{\gcd(15,10)}}=50\cdot2\cdot(2)^{-1}_{\bmod 3}=200
\end{flalign*}
The nonzero entries of the third basis element $\h_2$ are calculated as follows:
\begin{flalign*}
h_3^{(2)} &=\lcm(6,\gcd(10,15))=30 \\
h_4^{(2)} &=30\left(\frac{\gcd(10)}{\gcd(15,10)}\right)\left(\frac{\gcd(10)}{\gcd(15,10)}\right)^{-1}_{\bmod \frac{15}{\gcd(15,10)}}=50\cdot2\cdot(2)^{-1}_{\bmod 3}=120
\end{flalign*}
The only  nonzero element of the final basis element $\h_3$ is $h_4^{(3)}=\lcm(15,10)=30$. Thus we have the following triangulation basis for the $4$-cycle with edge labels $\{2,6,10,15\}$: $\h_0=(1,1,1,1)$, $\h_1=(0,2,15,200)$, $\h_3=(0,0,30,120)$, and \mbox{$\h_4=(0,0,0,30)$}.


\section{The King Splines}\label{kingsec}
In this section we define King splines on $n$-cycles and prove that they form a basis for the set of splines.

\begin{definition}[King splines]
\label{KingSplines}
Fix a cycle with edge-labels $(C_n,L)$ and assume $\ell_{n-1}$ and $\ell_n$ relatively prime. The King splines on $(C_n,L)$ are the vectors

\[K_0= \begin{pmatrix} 1 \\ 1 \\ \vdots \\ 1 \\ 1 \\ 1 \end{pmatrix},  
K_1= \begin{pmatrix} k_1 \\ \ell_1 \\ \vdots \\ \ell_1 \\ \ell_1 \\ 0 \end{pmatrix},
K_2 = \begin{pmatrix} k_2 \\ \ell_2 \\ \vdots \\ \ell_2 \\ 0 \\ 0 \end{pmatrix}, ... ,
 K_{n-1}= \begin{pmatrix} k_{n-1} \\ 0 \\ \vdots \\ 0 \\ 0 \\ 0 \end{pmatrix}\]
where
\[ k_i= \begin{cases}
\gthree & \text{for } 1 \leq i \leq n-2 \\
\ell_{n-1}\ell_n & \text{for } i=n-1.  
\end{cases} \]
By convention, we call $K_0$ the trivial King spline. 
\end{definition}

As our terminology suggests, the King splines are in fact splines.

\begin{theorem}
Let $n\geq 3$. Fix a cycle with edge-labels $(C_n,L)$ with $\ell_{n-1}$ and $\ell_n$ relatively prime. The King splines $K_0 ,...,K_{n-1}$ are splines on $(C_n,L)$. 
\end{theorem}

\begin{proof}

First we note that the trivial King spline $K_0$ is the same as the trivial spline $\G_0$ which is indeed a spline on $(C_n,L)$ by Proposition \ref{TrivialSplines}.

Consider an arbitrary King spline $K_i = (0,\ldots,0,\ell_i,\ldots,\ell_i,k_{n-1})$ where \mbox{$1 \leq i \leq n-2$}. It has zero for its first $i$ entries, $\ell_i$ for entries $i+1$ to $n-1$, and $k_{n-1}$ for its last entry. We want to show that $K_i$ is a spline on $(C_n,L)$. Note that zero is congruent to itself modulo any integer, so in particular the following congruences are satisfied: 

\begin{equation}\label{neweq1}
\begin{cases}
 0 \equiv 0 \bmod \ell_j & \text{for } 1 \leq j \leq i-1 \\
\end{cases} 
\end{equation}

Also, since the integer $\ell_i$ is congruent to zero modulo $\ell_i$ we have

\begin{equation}\label{neweq2}
\ell_i \equiv 0 \bmod \ell_i 
\end{equation}

The integer $\ell_i$ is congruent to itself modulo any integer, so in particular the following congruences are satisfied: 

\begin{equation}\label{neweq3}
\begin{cases}
 \ell_i \equiv \ell_i \bmod \ell_{j} & \text{for } i+1 \leq j \leq n-2 \\
\end{cases} 
\end{equation}

Finally we know $k_i=\gthree$ satisfies the following two congruences
\begin{equation}\label{neweq4} \begin{cases}
k_i \equiv \ell_i \bmod \ell_{n-1} \\
k_i \equiv 0 \bmod \ell_n \\
\end{cases} 
\end{equation}
by Proposition \ref{modprop}. Collect the congruences in \ref{neweq1}, \ref{neweq2}, \ref{neweq3}, and \ref{neweq4} into a single system of congruences. This system represents the edge conditions on $(C_n,L)$. The vector $K_i$ satisfies all of these congruences so $K_i$ is a spline on $(C_n,L)$.

Now consider the vector $K_{n-1}=(0,...,0,k_{n-1})$.  Zero is congruent to itself modulo any integer, so the following system of congruences is satisfied: 
\begin{equation}\label{eq4}
\begin{cases}
 0 \equiv 0 \bmod \ell_j  & \text{for } 1 \leq j \leq n-2.
\end{cases}
\end{equation}
Since $k_{n-1}=\ell_{n-1} \ell_n$ we know 
\begin{equation}\label{neweq5}
\begin{cases}
k_{n-1} \equiv 0 \bmod \ell_{n-1} \\
k_{n-1} \equiv 0 \bmod \ell_n \\
\end{cases}
\end{equation}
Collect the congruences in \ref{eq4} and \ref{neweq5} into a single system. This system represents the edge conditions on $(C_n,L)$. The vector $K_{n-1}$ satisfies all of these congruences so $K_{n-1}$ is a spline on $(C_n,L)$.

Thus we have that $K_i$ is a spline for all $0 \leq i \leq n-1$ as desired.

\end{proof}

Now that we know the King splines are splines, we confirm that they form a basis.

\begin{theorem}
Fix a cycle with edge labels $(C_n,L)$ with $\ell_{n-1}$ and $\ell_n$  relatively prime. The set of King splines $K_0,...,K_{n-1}$ forms a basis for the set of splines on $(C_n,L)$.
\end{theorem}

\begin{proof}
The set of smallest flow-up classes $\G_0,...,\G_{n-1}$ form a basis for the set of splines on $(C_n,L)$ by Theorem \ref{HMRbasis}. We constructed the King splines so that the leading entry $\K_i$ equals the leading entry of $\G_i$ for $0\leq i \leq n-1$. Thus the set of King splines $K_0,...,K_{n-1}$ forms a basis for the set of splines on $(C_n,L)$ by Corollary \ref{bccor}.
\end{proof}


\section{Multiplication Tables}\label{multsec}

The fact that we have simple explicit formulas for the entries of the King basis is a powerful computational tool. In this section we use the King basis to write the product of any pair of basis elements as a linear combination of basis elements. This kind of calculation is important in geometry and topology, which use splines over polynomial rings to describe cohomology rings.

\subsection{Multiplication Tables for $n$-Cycles on the King Basis} When multiplying splines the operation is performed component-wise. Consider the King basis on a given n-cycle.  

Since the entries in the trivial spline $K_0$ are all ones, multiplying any spline $K_i$ (with $0\leq i \leq n-1$) by $K_0$ simply yields $K_i$. The following theorem gives us the product of any pair of non-trivial King splines.

\bigskip

\begin{theorem}
For arbitrary $K_i, K_j$ with $i, j \neq 0$ and $i \leq  j$, we have the product
 \[ K_i K_j =  l_i K_j + \frac{k_j (k_i - l_i)}{k_{n-1}}K_{n-1}.\]
\end{theorem}

\begin{proof}

We give a proof by construction. 

 Consider arbitrary  basis elements  $K_i$ and $K_j$ with $i, j \neq 0$ and $i \leq  j$. Their product $K_i K_j$ has zeros up to the $j^{th}$ entry.  The entries numbered $j + 1$ through $n - 1$ are  $\ell_i \cdot \ell_j$. The last entry is $k_i \cdot k_j$. 

Note that $\ell_i \cdot K_j$ has zeros for the first $j$ entries, $\ell_i \cdot \ell_j$ from entries $j + 1$ to $n - 1$, and $\ell_i \cdot k_j$ for the $n^{th}$ entry. This is almost exactly the product $K_i K_j$. However we want this last entry to be $k_i \cdot k_j$. Adding $\frac{k_j (k_i - l_i)}{k_{n-1}}K_{n-1}$ gives the desired result.

\smallskip

 Thus for $K_i K_j$ with $i, j \neq 0$ and $i \leq  j$ we have  
 \begin{center}
 $K_i K_j= \ell_i K_j + \frac{k_i k_j - l_i k_j}{k_{n-1}}K_{n-1} =  \ell_i K_j + \frac{k_j (k_i - l_i)}{k_{n-1}}K_{n-1}$
\end{center}

\bigskip

 Since we are working in the integers, our last step is to prove that the coefficient $$\frac{k_j (k_i - \ell_i)}{k_{n-1}}$$  is indeed an integer. We know $k_i \equiv \ell_i \bmod \ell_{n-1}$ because $K_i$ is a spline. Say \mbox{$k_i - \ell_i = p\ell_{n-1}$} for some $p \in \Z$. Similarly, we know $k_j \equiv 0 \bmod \ell_n$ because $K_j$ is a spline. Say $k_j = q \ell_n$ for some $q \in \Z$. By definition we have  $k_{n-1} = \ell_{n-1}\ell_n$. Plugging these values into the expression $\frac{k_i k_j - l_i k_j}{k_{n-1}}$ yields the following:

\begin{equation*}
\frac{k_j (k_i - \ell_i)}{k_{n-1}}  =  \frac{(q\ell_n)(p\ell_{n-1})}{\ell_{n-1}\ell_n} = pq
\end{equation*}

 Thus $\frac{k_j (k_i - \ell_i)}{k_{n-1}}$ is always an integer. 

\end{proof}

 Note that the product $K_i K_{n -1}$ for any $i \leq n - 1$ simplifies significantly. 

\begin{corollary}
Choose any $i \neq 0$. Then $K_i K_{n-1} = k_iK_{n-1}.$
\end{corollary}

\begin{proof}
We apply the formula for the product $K_iK_j$ to the particular case where $j = n-1$ and simplify:

\begin{equation*}
K_i K_{n-1}  =  \ell_i K_{n-1} + \frac{k_{n - 1}(k_i - \ell_i)}{k_{n-1}}K_{n-1}  =  k_i K_{n-1}
\end{equation*}

\end{proof}
\bigskip

\noindent For example consider the 5-cycle with edge labels $\{3, 4, 8, 2, 5\}$. The King basis on a 5-cycle with these labels looks like the following: 

\bigskip 

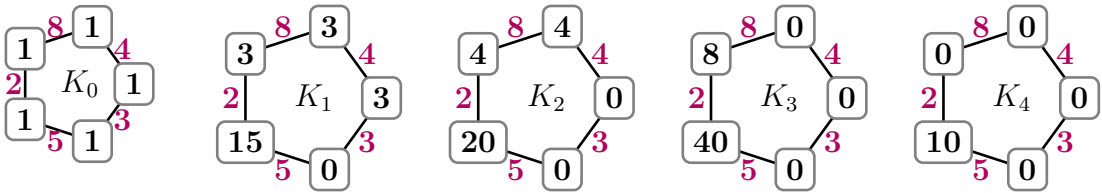
\begin{multicols}{5}

\begin{center}
\begin{tikzpicture}
	
	\begin{scope}[scale=.8]
		\pgfmathsetmacro{\r}{1}

		\node at (.1,0) {$K_0$};

		\draw[edge] (-72:\r)--(0:\r);
		\draw[edge] (0:\r)--(72:\r);
		\draw[edge] (72:\r)--(144:\r);
		\draw[edge] (144:\r) -- (-144:\r);	
		\draw[edge] (-144:\r) -- (-72:\r);	

		\node[edgelabel] at (-108:\r) {$5$};
		\node[edgelabel] at (-36:\r) {$3$};
		\node[edgelabel] at (36:\r) {$4$};
		\node[edgelabel] at (108:\r) {$8$};
		\node[edgelabel] at (180:\r) {$2$};
		
		\node[vertex] at (-72:\r) {$1$};
		\node[vertex] at (0:\r) {$1$};
		\node[vertex] at (72:\r) {$1$};
		\node[vertex] at (144:\r) {$1$};
		\node[vertex] at (-144:\r) {$1$};
	\end{scope}
\end{tikzpicture}
\end{center}

\columnbreak
\begin{center}
\begin{tikzpicture}
	
	\begin{scope}
		\pgfmathsetmacro{\r}{1}

		\node at (.1,0) {$K_1$};

		\draw[edge] (-72:\r)--(0:\r);
		\draw[edge] (0:\r)--(72:\r);
		\draw[edge] (72:\r)--(144:\r);
		\draw[edge] (144:\r) -- (-144:\r);	
		\draw[edge] (-144:\r) -- (-72:\r);	

		\node[edgelabel] at (-108:\r) {$5$};
		\node[edgelabel] at (-36:\r) {$3$};
		\node[edgelabel] at (36:\r) {$4$};
		\node[edgelabel] at (108:\r) {$8$};
		\node[edgelabel] at (180:\r) {$2$};
		
		\node[vertex] at (-72:\r) {$0$};
		\node[vertex] at (0:\r) {$3$};
		\node[vertex] at (72:\r) {$3$};
		\node[vertex] at (144:\r) {$3$};
		\node[vertex] at (-144:\r) {$15$};
	\end{scope}
\end{tikzpicture}
\end{center}

\columnbreak
\begin{center}
\begin{tikzpicture}
	
	\begin{scope}
		\pgfmathsetmacro{\r}{1}

		\node at (.1,0) {$K_2$};

		\draw[edge] (-72:\r)--(0:\r);
		\draw[edge] (0:\r)--(72:\r);
		\draw[edge] (72:\r)--(144:\r);
		\draw[edge] (144:\r) -- (-144:\r);	
		\draw[edge] (-144:\r) -- (-72:\r);	

		\node[edgelabel] at (-108:\r) {$5$};
		\node[edgelabel] at (-36:\r) {$3$};
		\node[edgelabel] at (36:\r) {$4$};
		\node[edgelabel] at (108:\r) {$8$};
		\node[edgelabel] at (180:\r) {$2$};
		
		\node[vertex] at (-72:\r) {$0$};
		\node[vertex] at (0:\r) {$0$};
		\node[vertex] at (72:\r) {$4$};
		\node[vertex] at (144:\r) {$4$};
		\node[vertex] at (-144:\r) {$20$};
	\end{scope}
\end{tikzpicture}
\end{center}

\columnbreak
\begin{center}
\begin{tikzpicture}
	
	\begin{scope}
		\pgfmathsetmacro{\r}{1}

		\node at (.1,0) {$K_3$};

		\draw[edge] (-72:\r)--(0:\r);
		\draw[edge] (0:\r)--(72:\r);
		\draw[edge] (72:\r)--(144:\r);
		\draw[edge] (144:\r) -- (-144:\r);	
		\draw[edge] (-144:\r) -- (-72:\r);	

		\node[edgelabel] at (-108:\r) {$5$};
		\node[edgelabel] at (-36:\r) {$3$};
		\node[edgelabel] at (36:\r) {$4$};
		\node[edgelabel] at (108:\r) {$8$};
		\node[edgelabel] at (180:\r) {$2$};
		
		\node[vertex] at (-72:\r) {$0$};
		\node[vertex] at (0:\r) {$0$};
		\node[vertex] at (72:\r) {$0$};
		\node[vertex] at (144:\r) {$8$};
		\node[vertex] at (-144:\r) {$40$};
	\end{scope}
\end{tikzpicture}
\end{center}

\columnbreak
\begin{center}
\begin{tikzpicture}
	
	\begin{scope}
		\pgfmathsetmacro{\r}{1}

		\node at (.1,0) {$K_4$};

		\draw[edge] (-72:\r)--(0:\r);
		\draw[edge] (0:\r)--(72:\r);
		\draw[edge] (72:\r)--(144:\r);
		\draw[edge] (144:\r) -- (-144:\r);	
		\draw[edge] (-144:\r) -- (-72:\r);	

		\node[edgelabel] at (-108:\r) {$5$};
		\node[edgelabel] at (-36:\r) {$3$};
		\node[edgelabel] at (36:\r) {$4$};
		\node[edgelabel] at (108:\r) {$8$};
		\node[edgelabel] at (180:\r) {$2$};
		
		\node[vertex] at (-72:\r) {$0$};
		\node[vertex] at (0:\r) {$0$};
		\node[vertex] at (72:\r) {$0$};
		\node[vertex] at (144:\r) {$0$};
		\node[vertex] at (-144:\r) {$10$};
	\end{scope}
\end{tikzpicture}
\end{center}

\end{multicols}

\bigskip

Let's multiply the elements $K_1$ and $K_3$. We obtain \\

\begin{tikzpicture}

	\begin{scope}
	\node at (-3,0) {$K_1 K_3 = $ };
	\end{scope}	

	\begin{scope}
		\pgfmathsetmacro{\r}{1}

		\node at (.1,0) {$K_1$};

		\draw[edge] (-72:\r)--(0:\r);
		\draw[edge] (0:\r)--(72:\r);
		\draw[edge] (72:\r)--(144:\r);
		\draw[edge] (144:\r) -- (-144:\r);	
		\draw[edge] (-144:\r) -- (-72:\r);	

		\node[edgelabel] at (-108:\r) {$5$};
		\node[edgelabel] at (-36:\r) {$3$};
		\node[edgelabel] at (36:\r) {$4$};
		\node[edgelabel] at (108:\r) {$8$};
		\node[edgelabel] at (180:\r) {$2$};
		
		\node[vertex] at (-72:\r) {$0$};
		\node[vertex] at (0:\r) {$3$};
		\node[vertex] at (72:\r) {$3$};
		\node[vertex] at (144:\r) {$3$};
		\node[vertex] at (-144:\r) {$15$};
	\end{scope}

	\begin{scope}
	\node at (.5,0) {\hspace{1in}$\times$\hspace{1in}};
	\end{scope}
	
	\begin{scope}[xshift=1.5in]
		\pgfmathsetmacro{\r}{1}

		\node at (.1,0) {$K_3$};

		\draw[edge] (-72:\r)--(0:\r);
		\draw[edge] (0:\r)--(72:\r);
		\draw[edge] (72:\r)--(144:\r);
		\draw[edge] (144:\r) -- (-144:\r);	
		\draw[edge] (-144:\r) -- (-72:\r);	

		\node[edgelabel] at (-108:\r) {$5$};
		\node[edgelabel] at (-36:\r) {$3$};
		\node[edgelabel] at (36:\r) {$4$};
		\node[edgelabel] at (108:\r) {$8$};
		\node[edgelabel] at (180:\r) {$2$};
		
		\node[vertex] at (-72:\r) {$0$};
		\node[vertex] at (0:\r) {$0$};
		\node[vertex] at (72:\r) {$0$};
		\node[vertex] at (144:\r) {$8$};
		\node[vertex] at (-144:\r) {$40$};
	\end{scope}

	\begin{scope}
	\node at (4.5,0) {\hspace{1in}$=$\hspace{1in}};
	\end{scope}
	
	\begin{scope}[xshift=3in]
		\pgfmathsetmacro{\r}{1}

		\draw[edge] (-72:\r)--(0:\r);
		\draw[edge] (0:\r)--(72:\r);
		\draw[edge] (72:\r)--(144:\r);
		\draw[edge] (144:\r) -- (-144:\r);	
		\draw[edge] (-144:\r) -- (-72:\r);	

		\node[edgelabel] at (-108:\r) {$5$};
		\node[edgelabel] at (-36:\r) {$3$};
		\node[edgelabel] at (36:\r) {$4$};
		\node[edgelabel] at (108:\r) {$8$};
		\node[edgelabel] at (180:\r) {$2$};
		
		\node[vertex] at (-72:\r) {$0$};
		\node[vertex] at (0:\r) {$0$};
		\node[vertex] at (72:\r) {$0$};
		\node[vertex] at (144:\r) {$24$};
		\node[vertex] at (-144:\r) {$600$};
	\end{scope}
\end{tikzpicture}

\bigskip

\noindent By the formula given above \[ K_1 K_3  = 3  K_3 + \frac{40 (15 - 3)}{10}K_4 = 3 K_3 + 48 K_4.\] Pictorially this solution is shown below.  \\

\noindent
\begin{tikzpicture}
	
	\begin{scope}
	\node at (-5,0) {$ 3K_3 + 48 K_4 = $};
	\end{scope}

	\begin{scope}
	\node at (-3.5,0) {$3$};
	\end{scope} 

	\begin{scope}[xshift=-.75in]
		\pgfmathsetmacro{\r}{1}

		\node at (.1,0) {$K_3$};

		\draw[edge] (-72:\r)--(0:\r);
		\draw[edge] (0:\r)--(72:\r);
		\draw[edge] (72:\r)--(144:\r);
		\draw[edge] (144:\r) -- (-144:\r);	
		\draw[edge] (-144:\r) -- (-72:\r);	

		\node[edgelabel] at (-108:\r) {$5$};
		\node[edgelabel] at (-36:\r) {$3$};
		\node[edgelabel] at (36:\r) {$4$};
		\node[edgelabel] at (108:\r) {$8$};
		\node[edgelabel] at (180:\r) {$2$};
		
		\node[vertex] at (-72:\r) {$0$};
		\node[vertex] at (0:\r) {$0$};
		\node[vertex] at (72:\r) {$0$};
		\node[vertex] at (144:\r) {$8$};
		\node[vertex] at (-144:\r) {$40$};
	\end{scope}

	\begin{scope}
	\node at (.5,0) {$+$ \hspace{.1in} $48$};
	\end{scope}	

	\begin{scope}[xshift=1in]
		\pgfmathsetmacro{\r}{1}

		\node at (.1,0) {$K_4$};

		\draw[edge] (-72:\r)--(0:\r);
		\draw[edge] (0:\r)--(72:\r);
		\draw[edge] (72:\r)--(144:\r);
		\draw[edge] (144:\r) -- (-144:\r);	
		\draw[edge] (-144:\r) -- (-72:\r);	

		\node[edgelabel] at (-108:\r) {$5$};
		\node[edgelabel] at (-36:\r) {$3$};
		\node[edgelabel] at (36:\r) {$4$};
		\node[edgelabel] at (108:\r) {$8$};
		\node[edgelabel] at (180:\r) {$2$};
		
		\node[vertex] at (-72:\r) {$0$};
		\node[vertex] at (0:\r) {$0$};
		\node[vertex] at (72:\r) {$0$};
		\node[vertex] at (144:\r) {$0$};
		\node[vertex] at (-144:\r) {$10$};
	\end{scope}

	\begin{scope}
	\node at (3.25,0) {\hspace{1in}$=$\hspace{1in}};
	\end{scope}
	
	\begin{scope}[xshift=2.5in]
		\pgfmathsetmacro{\r}{1}

		\draw[edge] (-72:\r)--(0:\r);
		\draw[edge] (0:\r)--(72:\r);
		\draw[edge] (72:\r)--(144:\r);
		\draw[edge] (144:\r) -- (-144:\r);	
		\draw[edge] (-144:\r) -- (-72:\r);	

		\node[edgelabel] at (-108:\r) {$5$};
		\node[edgelabel] at (-36:\r) {$3$};
		\node[edgelabel] at (36:\r) {$4$};
		\node[edgelabel] at (108:\r) {$8$};
		\node[edgelabel] at (180:\r) {$2$};
		
		\node[vertex] at (-72:\r) {$0$};
		\node[vertex] at (0:\r) {$0$};
		\node[vertex] at (72:\r) {$0$};
		\node[vertex] at (144:\r) {$24$};
		\node[vertex] at (-144:\r) {$600$};
	\end{scope}
\end{tikzpicture}

\begin{remark}
The same argument can be used to give the multiplication table for arbitrarily labeled 3-cycles using the \newspline \ basis (Def \ref{GSFUCdef}, Thrm \ref{GSFUCthm}). Given the basis elements $\h_0, \h_1,$ and $\h_2$ we have the following table
$$\h_0 = \left(\begin{array}{c} 1\\1\\1\end{array}\right),  \h_1 = \left(\begin{array}{c} h_3^{(1)}\\ h_2^{(1)}\\0\end{array}\right),  \h_2 = \left(\begin{array}{c} h_3^{(2)}\\ 0 \\0\end{array}\right)$$
\begin{table}[ht]
\centering
\begin{tabular}{ccccc}
~ & \vline & $\h_0$ & $\h_1$ & $\h_2$ \\
\hline
$\h_0$ & \vline & $\h_0$ & $\h_1$ & $\h_2$ \\
$\h_1$ & \vline & $\h_1$ & $h_{2}^{(1)}\h_1 + \Phi \h_2$ & $h_3^{(1)}\h_2$ \\
$\h_2$ & \vline & $\h_2$ & $h_3^{(1)}\h_2$ & $h_3^{(2)}\h_2$ \\
\end{tabular}
\end{table}

where $\Phi = \frac{h_3^{(1)}(h_3^{(1)} -h_2^{(1)})}{h_3^{(2)}}$.
\end{remark}

Unlike with the King basis, we do not have nice formulas for entries of the \newspline \ basis. This leads to the following open question.

\begin{question}
Is there a positive or combinatorial formula for the multiplication table of general $n$-cycles ($i.e.$ not alternating sums from successively correcting each spline entry)?
\end{question}


\end{document}